\documentclass[11pt,letterpaper]{article}

\usepackage[utf8x]{inputenc}
\usepackage[T1]{fontenc}
\usepackage[english]{babel}

\usepackage{amsmath}
\usepackage{amsthm}
\usepackage{amssymb}
\usepackage{amsfonts}

\usepackage{epsfig}
\usepackage{graphicx}
\usepackage{subcaption}
\usepackage[linesnumbered,ruled,vlined]{algorithm2e}
\usepackage{xcolor}

\usepackage{hyperref}

\usepackage[inline]{enumitem}

\usepackage[margin=1.15in]{geometry}
\usepackage[toc,page]{appendix}

\usepackage{authblk}% Para la afiliación

\newtheorem{thm}{Theorem}
\newtheorem{cor}[thm]{Corollary}
\newtheorem{lem}[thm]{Lemma}
\newtheorem{prop}[thm]{Proposition}

\newcommand{\rt}{ {\bf{0}} }
\newcommand{\pr}[1]{ \mathbb{P}\left( #1 \right) }
\newcommand{\EE}[1]{\mathbb{E}\left( #1 \right)}

\newcommand{\rtree}[1]{\mathbb{T}_{#1}}

\newcommand{\RR}{\mathbb{R}}
\newcommand{\GG}{\mathcal{G}}
\newcommand{\PP}{\mathcal{P}}
\newcommand{\Aa}[1]{\mathcal{A}_{#1}}
\newcommand{\ZZ}{\mathbb{Z}}
\newcommand{\Zd}{\mathbb{L}^2}
\newcommand{\Zda}{\mathbb{L}^2_{alt}}
\newcommand{\NN}{\mathbb{N}}
\newcommand{\pprmf}[3]{\Pi_{RMF} \left\langle #1,#2,#3 \right\rangle}
\newcommand{\pprmfd}[2]{\Pi_{RMF} \left\langle #1,#2 \right\rangle}
\newcommand{\ppber}[3]{\Pi_{Ber} \left\langle #1,#2,#3 \right\rangle}
\newcommand{\ppberd}[2]{\Pi_{Ber} \left\langle #1,#2 \right\rangle}
\newcommand{\set}[1]{\left\lbrace #1\right\rbrace}

\newcommand{\hyc}{\mathcal{Q}}
\newcommand{\hycn}[1]{\mathbb{Q}_{#1}}
\newcommand{\floor}[1]{\left\lfloor #1 \right\rfloor}

\newcommand{\tupla}[1]{\left( #1 \right)}
\newcommand{\minC}[2]{\emph{MinMass}(#1,#2)}
\newcommand{\maxC}[2]{\emph{MaxMass}(#1,#2)}
\newcommand{\stdom}[2]{#1 \succeq #2}

\title{RMF accessibility percolation on oriented graphs.}

\author[1,2]{Frank Duque}
\author[1]{Daniel Ramirez-Gomez}
\author[1]{Alejandro Rold\'an-Correa}
\author[1]{Leon A. Valencia}
\affil[1]{Instituto de Matem\'aticas, Universidad de Antioquia, Colombia}
\affil[2]{Escuela de Matem\'aticas, Universidad Nacional de Colombia, Colombia}

\begin{document}
\maketitle

\begin{abstract}
Accessibility percolation is a new type of percolation problem inspired by evolutionary biology: a random number, called its fitness, is assigned to each vertex of a graph, then a path in the graph is accessible if fitnesses are strictly increasing through it.
In the Rough Mount Fuji (RMF) model the fitness function is defined on the graph as $\omega(v)=\eta(v)+\theta\cdot d(v)$, where $\theta$ is a positive number called the drift, $d$ is the distance to the source of the graph and $\eta(v)$ are i.i.d. random variables. In this paper we determine values of $\theta$ for having RMF accessibility percolation on the hypercube and the two-dimensional lattices $\Zd$ and $\Zda$.
\end{abstract}

\maketitle

\section{Introduction} \label{sec:introduction}

Let $G$ be a graph with a distinguished vertex $v_0$, and let $\omega: V(G)\rightarrow \RR$ be a function that assigns to each vertex of $G$ a real number called fitness of the vertex. 
 We say that a path \[P=v_0\rightarrow v_1\rightarrow \cdots \rightarrow v_k\] is accessible for $\omega$, if $\omega(v_{i+1})>\omega(v_i)$ for $i=0,1,\ldots,k-1$.\\

The study of  long accessible paths is related to the classical model for the evolution of an organism that involves mutation and genetic selection, see \cite{WDDH2006, WWC2005}. Nowak and Krug \cite{NK2013} called \textit{accessibility percolation} to the existence of such long accessible paths.  For a complete review on accessibility percolation and its biological motivation see \cite{10.2307/43864013, krug2021accessibility, RZ2013}. The fitness function $\omega$ may be defined in particular random ways, obtaining different accessibility percolation models. The Rough Mount Fuji (RMF) model is obtained when $\omega_{\eta,\theta}(v)=\eta(v)+\theta\cdot d(v)$, where $\eta(v)$ is a collection of i.i.d random variables, $d(v)$ denotes the distance from $v$ to $v_0$ (the source of $G$). The RMF model was first proposed in \cite{RMF-DEF}; see also \cite{JKAJ2011,JGJ2010, Martinsson2014}. Observe that RMF model favors the accessibility percolation when $\theta>0$,  since the drift factor will tend to increase fitness of the vertices throughout paths.

Much of the current accessibility percolation theory is on the unconstrained house of cards (HoC) model, this is a particular case of the RMF model when $\theta=0$; some of those works are \cite{JKAJ2011, FK2012,  Martinsson2014,  kingman1978, Li2017, NK2013, ParkKrug2008, RZ2013}.  About the RMF model, there are fewer results \cite{JKAJ2011,FK2012,  JGJ2010, Martinsson2014,  Neidhart2014, NK2013, SMK2013}.  In this paper, we study the RMF model on the hypercube and the two-dimensional lattices $\Zd$ and $\Zda$.\\

Throughout this paper, $G$ denotes a graph where the edges always point to vertices further away from a vertex $v_0$, called source.
Given a distribution function $F$, we denote by $\eta_{G,F}$ the set of independent and identically distributed random variables $\eta_{G,F}=\set{\eta(v)}_{v\in G}$ where each $\eta(v)$ has distribution $F$. When $G$ and $F$ are understood we write $\eta$ instead of $\eta_{G,F}$.
We adopt the notations used in \cite{Martinsson2014} for $\sim$, $\gtrsim$, $\lesssim$, $O$ and $\Omega$ (see Appendix~\ref{sec:notation}).

Let $G$ be a finite graph. We say that a path $P$ in $G$ is \emph{long}, if it starts at the source and ends at a sink.
Let \[ \pprmfd{G}{F,\theta}= \text{Probability of having a long accessible path for $\omega_{\eta,\theta}$ in $G$}. \]
Given a sequence of finite graphs $\GG=(G_1,G_2,\ldots)$, we say that there is \emph{RMF (accessibility) percolation} in $(\GG,F,\theta)$, if  $\lim_{n\rightarrow \infty}\pprmfd{G_n}{F,\theta}=1$.
We say that $\nu_p=\{\nu_p(v)\}_{v\in V(G)}$ is a Bernoulli process, if $\nu_p$ is a collection of i.i.d random variables and
$\pr{\nu_p(v)=1}=p$ for each $v\in V(G)$. Given a Bernoulli process $\nu_p$, the path $P=v_0\rightarrow v_1\rightarrow \cdots \rightarrow v_k$ is called open if  $\nu_p(v_i)=1$ for $i=0,1,\ldots,k$. Let
\[\ppberd{G}{p}=\text{Probability of having a long open path in } \nu_p.\]
Given a sequence of finite graphs $\GG=(G_1,G_2,\ldots)$, we say that there is \emph{Bernoulli percolation} in $(\GG,p)$, if  $\lim_{n\rightarrow \infty}\ppberd{G}{p}=1$.

The support of this paper is Lemma~\ref{lem:basic_coupling}.
This result is obtained by coupling between the RMF model with the Bernoulli process, that is, by joint constructing the models in the same probability space, in order to gain relations between them
(for more details on concepts of coupling, see \cite{thorisson2000coupling}).
As consequence: we bound the probability of having RMF percolation with the probability of having Bernoulli percolation, see Theorem \ref{thm:coupling}; and we bound the number of accessible paths with  the number of open paths, see Proposition~\ref{prop:coupling}. In particular, we use Theorem~\ref{thm:coupling}, in order to determine bounds for $\theta$ for having RMF accessibility percolation on the hypercube and the two-dimensional lattices $\Zd$ and $\Zda$.

\begin{thm}\label{thm:coupling}
 Let $F$ be any distribution,  $\PP$ be a set of paths in $G$ and $\theta>0$. Then, for any  $p\in [0,1]$ and $x\in \RR$ such that $p\leq F(x+\theta)-F(x)$,
 \[\pprmf{G}{\PP}{F,\theta}\geq\ppber{G}{\PP}{p}; \]
 where $\ppber{G}{\PP}{p}$ denotes the probability of having an open path in $\PP$ for the  Bernoulli process $\nu_p$, and
 $\pprmf{G}{\PP}{F,\theta}$  denotes the probability of having an  accessible path in $\PP$ for the RMF process $(G,\omega_{\eta,\theta})$.
\end{thm}

Stochastic dominance is a partial order between random variables: Given two random variables $X$ and $Y$ we say that $X$ stochastically dominates to $Y$ if, for all $x\in\mathbb{R}$,
\[\pr{X\geq x}\geq \pr{Y \geq x}.\]
This is denoted by $\stdom{X}{Y}$. For more details on concepts of stochastic dominance, see \cite{thorisson2000coupling}.

\begin{prop}\label{prop:coupling}
 Let $F$ be any distribution,  $\PP$ be a set of paths in $G$ and $\theta>0$. Then, for any   $p\in [0,1]$ and $x\in \RR$ such that $p\leq F(x+\theta)-F(x)$,
 \[\stdom{X_{RMF}}{X_B};\]
 where $X_B$ denotes the number of open paths of $\PP$ on the Bernoulli process $\nu_p$, and  $X_{RMF}$ denotes the number of accessible paths of $\PP$ on the RMF process $\tupla{G,\omega_{\eta,\theta}}$.
\end{prop}

\subsection{Results on the Hypercube}

Consider the graph whose vertices are all the binary strings of length $n$, and there is an edge between a pair of vertices if they differ in exactly one bit; this graph is known as the hypercube. We denote by $\hycn{n}$ the oriented hypercube where edges are always directed toward the vertex with the greater number of ones.

For $G_{n}=\hycn{n}$, the source is $\rt=\lbrace 0 \rbrace^n$, the sink is ${\bf 1}= \lbrace 1 \rbrace^n$, and
$\pprmfd{\hycn{n}}{F,\theta}$ denotes the probability of having an accessible path for $\omega_{\eta,\theta}$, that starts at $\rt$ and ends at ${\bf 1}$.
Let $\hyc=\tupla{\hycn{1},\hycn{2},\hycn{3},\ldots}$. Thus there is RMF accessibility percolation in $\tupla{\hyc, F,\theta}$, if  $\lim_{n\rightarrow \infty}\pprmfd{\hycn{n}}{F,\theta}=1$.

In \cite{Martinsson2014}, the authors obtain a Stochastic Domination between the number of accessible paths in the RMF model and the number of open paths in the Bernoulli process, in the case where the probability distribution of $F$ has continuous p.d.f. on its support and it is connected. From this they determine sufficient conditions for having RMF accessibility percolation. See Theorem~\ref{thm:hyc_Martinson}.

\begin{thm}[\cite{Martinsson2014}]\label{thm:hyc_Martinson}
 Let $F$ be any probability distribution whose p.d.f. is continuous on its support and whose support is connected. Let $\theta_n$ be any strictly positive
function of $n$ such that $n\theta_n \rightarrow \infty$ as $n\rightarrow \infty$. Then  $\pprmfd{\hycn{n}}{F,\theta_n}$ tends to one as $n\rightarrow \infty$.
\end{thm}

In the following result, we establish another version of Theorem \ref{thm:hyc_Martinson}, without  constraints on the distribution used in the RMF model.

\begin{thm} \label{thm:hypercube}
 Let $F$ be any distribution. Let $\theta_n$ be any strictly positive
function of $n$ such that $n\theta_n \rightarrow \infty$ as $n\rightarrow \infty$. If $\rt$ is labeled with $-\infty$ and ${\bf 1}$ is labeled with $\infty$ then $\pprmfd{\hycn{n}}{F,{\theta_n}}$ tends to one as $n\rightarrow \infty$.
\end{thm}

Given a distribution $F$ with  p.d.f. $f$ and a constant $\theta$, provided it exist, we denote by $\minC{F}{\theta}$ the minimal probability mass of $f$ on intervals with length $\theta$:
\[\minC{F}{\theta}=\min\set{\int_{I\subset Supp(f)}f(x)d(x): I \text{ is an interval with length } \theta}.\]

Given a distribution $F$ and a constant $\theta$, we denote by $\maxC{F}{\theta}$ the  supremum of the  probability mass of $F$ on intervals with length $\theta$:
\[\maxC{F}{\theta}=\sup\set{F(x+\theta)-F(x): x\in\RR}.\]

As a consequence of the construction used in \cite{Martinsson2014}, in order to prove Theorem~\ref{thm:hyc_Martinson},
the authors bound the number of long paths in the hypercube; see Proposition~\ref{prop:hyc_Martinson}. Using Proposition~\ref{prop:coupling},  we obtain another version of Proposition~\ref{prop:hyc_Martinson}, without constraints on the distribution used in the RMF model.

\begin{prop}[\cite{Martinsson2014}]\label{prop:hyc_Martinson}
 Suppose that $F$ has p.d.f. $f$ with connected and bounded support and $f$ is continuous on its support. Let $\theta>0$ be a constant and let $X=X(n)$ be the number of long accessible paths in $(\hycn{n},\omega_{\eta,\theta})$. Then
 \[X\gtrsim n!\cdot C^{n-1}\]
 where $C=\minC{F}{\theta/2}$.
\end{prop}

\begin{prop} \label{prop:hypercube:negligible01}
 Let $F$ be any distribution, let $\theta>0$ be constant, and let $X=X(n)$ be the number of long accessible paths in $(\hycn{n},\omega_{\eta,c})$.
 If $\rt$ is labeled with $-\infty$ and ${\bf 1}$ is labeled with $\infty$ then \[X\gtrsim n!\cdot C^{n-1}\] where $C=\maxC{F}{\theta}$.
\end{prop}

 We also modify the proof of Proposition~\ref{prop:hyc_Martinson} introduced in \cite{Martinsson2014}, obtaining another coupling and improving the (asymptotic) bound for the number of long paths.

\begin{prop}\label{prop:hypercube:nicefdp}
 Let $F$ be any distribution, let $\theta>0$ be constant, and let $X=X(n)$ be the number of long accessible paths in $(\hycn{n},\omega_{\eta,\theta})$.  If $F$ has p.d.f. $f$ with connected and bounded support and $f$ is continuous on its support, then
  \[X \sim \Omega(n!\cdot C^{n-1})\]
 where $C=\maxC{F}{\theta}$.
\end{prop}

\subsection{Results on infinite graphs}

Let $G$ be an infinite graph. We say that a path $P$ in $G$ is long, if it starts at the source and is infinite.
Like the finite case, we denote by $\pprmfd{G}{F,c}$ the probability of having a long accessible path for $\omega_{\eta,c}$ in $G$. Unlike the finite case, we say that there is RMF (accessible) percolation in $(G,F,c)$ if $\pprmfd{G}{F,c}>0$.

Some infinite graphs that had been of interest for the study of directed percolation are regular trees, $\mathbb{L}^n$ and $\mathbb{L}^n_{alt}$ \cite{DirPercZd2002}. A $d$-directed regular tree is a tree where each vertex has $d$ children. $\Zd$ is the graph obtained from the first quadrant of the two–dimensional lattice, whose bonds are oriented in the positive x and y directions.
$\Zda$ is the graph whose vertices are the $(x,y)\in \ZZ^2$ such that $y\geq 0$ and $|x|\leq y$; and there is a directed edge from $(x_1,y_1)$ to $(x_2,y_2)$ if $y_2=y_1+1$ and $|x_1-x_2|\leq 1$. See Figure~\ref{fig:L2yL2a} for illustrations of $\Zd$ and $\Zda$.

\begin{figure} [htb]
	\begin{subfigure}{.5\textwidth}
		\centering
		\includegraphics[width=.54\linewidth]{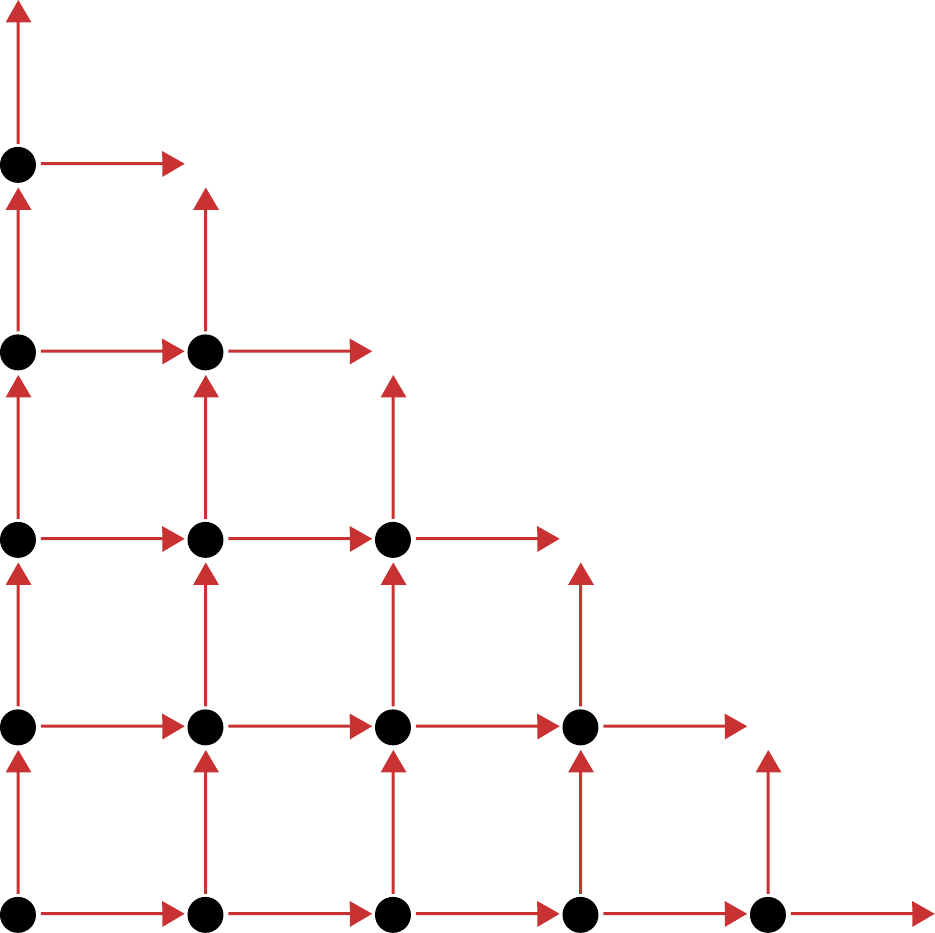}
		\caption{$\Zd$}
		\label{fig:L2}
	\end{subfigure}%
	\begin{subfigure}{.5\textwidth}
		\centering
		\includegraphics[width=.9\linewidth]{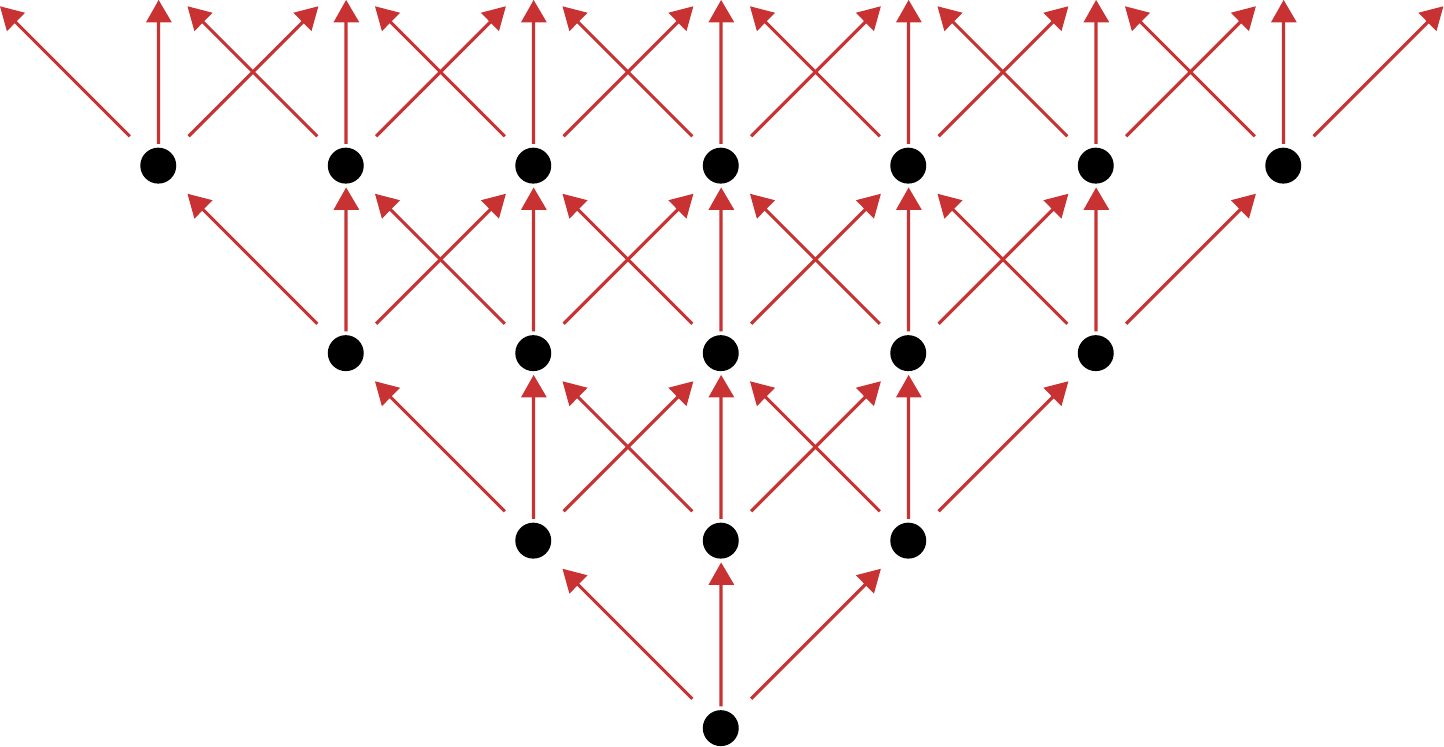}
		\caption{$\Zda$}
		\label{fig:L2a}
	\end{subfigure}%
	\caption{Ilustrations of $\Zd$ and $\Zda$. }
	\label{fig:L2yL2a}
\end{figure}

In Section~\ref{sec:inf:graphs}, we consider the problem of determining a threshold for having RMF percolation on infinite graphs: regular trees, $\Zd$ and $\Zda$. We determine bounds for $\theta$ for having RMF percolation in $(\Zd,F,\theta)$ and in $(\Zda,F,\theta)$; and we show computational results, from which we estimate the threshold for having RMF percolation in regular trees, $\Zd$ and $\Zda$.

This paper is organized as follows. First, in Section~\ref{sec:mainLema}, we build a coupling between the RMF model and the Bernoulli process, from which Theorem~\ref{thm:coupling} and Proposition~\ref{prop:coupling} are concluded. Then we prove the results on the hypercube: we prove Theorem~\ref{thm:hypercube} and Proposition~\ref{prop:hypercube:negligible01} in Section~\ref{sec:proofs:hypercube:01}, and we prove Proposition~\ref{prop:hypercube:nicefdp} in Section~\ref{sec:proofs:hypercube:nicefdp}. Finally, in Section~\ref{sec:inf:graphs} we prove the results on infinite graphs.

\section{A coupling with site percolation} \label{sec:mainLema}

In this section we prove that the RMF model is bounded by a Bernoulli process. First in Lemma~\ref{lem:basic_coupling}, we build a coupling between the RMF model and the Bernoulli process, where open paths are accessible paths. Then, Theorem~\ref{thm:coupling} and Proposition~\ref{prop:coupling} are obtained as consequence of Lemma~\ref{lem:basic_coupling}.

\begin{lem}\label{lem:basic_coupling}
 Let $\theta>0$ and let $\omega_{\eta,\theta}$ be a RMF model on a graph $G$. Let $F$ be the distribution of $\eta$.
 Then, for any $p\in [0,1]$ and $x\in \RR$ such that $p\leq F(x+\theta)-F(x)$, there exists a Bernoulli process $\nu_p$, in the same probability space than $\omega_{\eta,\theta}$, such that open paths in $\nu_p$ are accessible paths in $\omega_{\eta,\theta}$.
\end{lem}

\begin{proof}
We may assume that there exist $p$ and $x_{\theta}$ such that $0<p= F(x_{\theta}+\theta)-F(x_{\theta})$, the case $p< F(x_{\theta}+\theta)-F(x_{\theta})$ follow from this one.
Let
\[\nu_p(v)=\begin{cases}
          1 & \text{ if } x_{\theta}<\eta(v)\leq x_{\theta}+\theta, \\
          0 & \text{ otherwise.}  \\
         \end{cases}
\]

Note that $\nu_p$ is a Bernoulli process. Let
\[P=v_0\rightarrow v_1\rightarrow \cdots \rightarrow v_k\] be an open path in $\nu_p$. For obtaining the coupling, it remains to prove that $P$ is also an accessible path in $\omega_{\eta,\theta}$. Let $0\leq i \leq k-1$. Since the edges of $G$ are always pointing to vertices further away from a vertex $v_0$, $d(v_{i+1})=d(v_{i})+1$; and as $P$ is open,  $x_{\theta}<\eta(v_{i})\leq x_{\theta}+\theta$ and  $x_{\theta}< \eta(v_{i+1})\leq x_{\theta}+\theta$. Thus

\begin{align*}
 \omega_{\eta,\theta}(v_{i+1})-\omega_{\eta,\theta}(v_{i})=&\eta(v_{i+1})-\eta(v_i)+\theta\cdot d(v_{i+1})- \theta\cdot d(v_i)\\
 =& \eta(v_{i+1})-\eta(v_i)+\theta\\
 >& -\theta+\theta=0.
\end{align*}

Therefore, open paths in $\nu_p$ are accessible paths in $\omega_{\eta,\theta}$
\end{proof}

Theorem~\ref{thm:coupling} and Proposition~\ref{prop:coupling} follows from Lemma~\ref{lem:basic_coupling}.

\section{Proof of Theorem~\ref{thm:hypercube} and Proposition~\ref{prop:hypercube:negligible01}}\label{sec:proofs:hypercube:01}

In this section we prove Theorem~\ref{thm:hypercube} and Proposition~\ref{prop:hypercube:negligible01}. First, in Corollary~\ref{lem:BernoulliHypercube}, we consider the case of having Bernoulli percolation when $\rt$ and ${\bf 1 }$ are open. Then, we use Theorem~\ref{thm:coupling} and Corollary~\ref{lem:BernoulliHypercube} in order to prove Theorem~\ref{thm:hypercube}, and we use Proposition~\ref{prop:coupling} and Corollary~\ref{lem:BernoulliHypercube} in order to prove Proposition~\ref{prop:hypercube:negligible01}.

The reason for choosing $\rt$ and ${\bf 1 }$ open is because: the probability of having open paths, is limited by the probability of having the vertices $\rt$ and ${\bf 1 }$ open; and we require some Bernoulli process  on $\hyc=\tupla{\hycn{1},\hycn{2},\hycn{3},\ldots}$, such that, the limit of the probabilities of having open paths tends to $1$.

\begin{cor}\label{lem:BernoulliHypercube}
 Let $\nu_{p_n}$ be a process on $\hycn{n}$, where $\rt$ and ${\bf 1}$ are present and where every other vertex is present with probability $p$. Let $\ppberd{\hycn{n}}{p_n}$ denotes the probability of having an open path of $\hycn{n}$ from $\rt$ to ${\bf 1}$. Let $Y=Y_{n,p_n}$ be the number of open paths from $\rt$ to ${\bf 1}$ where  $\nu_{p_n}$ is as above. Then
 \begin{enumerate}
  \item \label{lem:BernoulliHypercube:1} $Y\sim\EE{Y}=n!p_n^{n-1}\approx \frac{ \sqrt{2\pi n} }{p_n} \left( \frac{np_n}{e}\right)^n$.
  \item \label{lem:BernoulliHypercube:2} If $np_n \rightarrow \infty$ as $n\rightarrow \infty$ then $Var(Y)=\EE{Y}+ o(\EE{Y}^2)$.
  \item \label{lem:BernoulliHypercube:3} If $np_n \rightarrow \infty$ as $n\rightarrow \infty$ then $\ppberd{\hycn{n}}{p_n}$ tends to one as $n\rightarrow \infty$.
 \end{enumerate}
\end{cor}

Corollary~\ref{lem:BernoulliHypercube} is proved in more general statements in Lemma~\ref{lem:2BernoulliHypercube}, such that Corollary~\ref{lem:BernoulliHypercube} can be obtained as a particular case of Lemma~\ref{lem:2BernoulliHypercube}, when $n_0=0$. Results \ref{lem:BernoulliHypercube:1} and \ref{lem:BernoulliHypercube:2} in Corollary~\ref{lem:BernoulliHypercube} were proved in \cite{Martinsson2014} (in Proposition~3.1); for an alternative proof of \ref{lem:BernoulliHypercube:3} in Corollary~\ref{lem:BernoulliHypercube}, see the proof of \ref{lem:BernoulliHypercube:3} of Lemma~\ref{lem:2BernoulliHypercube}.

\begin{proof}[\bf Proof of Proposition~\ref{prop:hypercube:negligible01}]
 Let  $X_{RMF}$ denotes the number of accessible long paths on the RMF process $\tupla{\hycn{n},\omega_{\eta,\theta}}$.
 Let $\epsilon>0$ and let $x_{\epsilon}\in \RR$ be such that
 \[C-\epsilon< F(x_{\epsilon}+\theta)-F(x_{\epsilon}) \leq C.\]

 If $X_{Ber(C-\epsilon)}$ denotes the number of open long paths on the Bernoulli process $\nu_{(C-\epsilon)}$, then: by Proposition~\ref{prop:coupling},
  \[\stdom{ X_{RMF} }{ X_{Ber(C-\epsilon)} };\]
and by (1) in Corollary~\ref{lem:BernoulliHypercube},
\[X_{Ber(C-\epsilon)}\gtrsim n!(C-\epsilon)^{n-1}.\]
Thus we have that, for any $\epsilon >0$  \[X_{RMF}\gtrsim n!(C-\epsilon)^{n-1}.\]

Therefore
\[X_{RMF}\gtrsim \sup_{\epsilon>0} \set{n!(C-\epsilon)^{h-1}}=n!C^{n-1}.\]
\end{proof}

\begin{proof}[{\bf Proof of Theorem~\ref{thm:hypercube}}]
 Let $F$ be any probability distribution. Let $\theta_n$ be any strictly positive
function of $n$ such that $n\theta_n \rightarrow \infty$ as $n\rightarrow \infty$. Here we assume that  $\rt$ is labeled with $-\infty$ and ${\bf 1}$ is labeled with $\infty$, and we prove that $\pprmfd{\hycn{n}}{F,\theta_n}$ tends to one as $n\rightarrow \infty$.

 Note that we may assume that $\theta_1>0$ and $\theta_n\leq \theta_1$ for $n\geq1$; other cases follow from this one.  Let $\epsilon_n=2^{\floor{\log_{2}\left( \frac{\theta_n}{\theta_1}\right)}}\theta_1$. Note that $\frac{\theta_n}{2}<\epsilon_n\leq \theta_n$  and $n\epsilon_n \rightarrow \infty$ as $n\rightarrow \infty$. As $\pprmfd{\hycn{n}}{F,\theta_n} \geq \pprmfd{\hycn{n}}{F,\epsilon_n}$ it is enough to prove that $\pprmfd{\hycn{n}}{F,\epsilon_n}$ tends to one as $n\rightarrow \infty$.

 Let $x_1$ be such that $F(x_1+\epsilon_1)-F(x_1)>0$ and let $m$ be such that \[F(x_1+\epsilon_1)-F(x_1)=m\epsilon_1>0.\]
 Let $p_n=m\epsilon_n$. Note that $np_n \rightarrow \infty$ as $n\rightarrow \infty$. Thus, by \ref{lem:BernoulliHypercube:3} in Corollary~\ref{lem:BernoulliHypercube}, $\ppberd{\hycn{n}}{p_n}$ tends to one as $n\rightarrow \infty$.

 By Theorem~\ref{thm:coupling}, if there exists a sequence $\set{x_n}_{n\in \NN}$ such $F(x_n+\epsilon_n)-F(x_n)\geq p_n$, then
 $\pprmfd{\hycn{n}}{F,\epsilon_n}\geq \ppberd{\hycn{n}}{p_n}$ and $\pprmfd{\hycn{n}}{F,\epsilon_n}$ tends to one as $n\rightarrow \infty$. Thus, it is enough to prove the existence of $x_n$ for $n=1,2,\ldots$, such that
 \begin{equation} \label{eq:proof:thm:hypercube}
  F(x_n+\epsilon_n)-F(x_n)\geq p_n.
 \end{equation}

 Note that, by definition of $x_1$ and $\epsilon_1$, $x_1$ satisfies Inequality~(\ref{eq:proof:thm:hypercube}); thus,
 $F$ has a mass of probability of at least $p_1=m\epsilon_1$ on the interval $(x_1,x_1+\epsilon_1]$.  In the following we recursively halve some intervals keeping the interval with more mass.

 Let $I_1=(x_1,x_1+\epsilon_1]$. Let $I_k$ be the sequence of intervals such that, if $I_k=(a,b]$ and $c=(a+b)/2$ then
 \[I_{k+1}=\begin{cases}
            (a,c] & \text{ if } F(c)-F(a)\geq F(b)-F(c),\\
            (c,b] & \text{ if } F(c)-F(a) < F(b)-F(c).\\
           \end{cases}
 \]
 It follows by recursion that, if $I_k=(a_k,b_k]$ then: $b_k-a_k=2^{-(k-1)}\epsilon_1$ and $F(b_k)-F(a_k)\geq  2^{-(k-1)}m\epsilon_1$.
 By definition of $\epsilon_n$, $\epsilon_1=\theta_1$ and, for each $n$ there exists $k$ such that $\epsilon_n=2^{-(k-1)}\epsilon_1$.
 Thus, for each $n$ there exist a $k$ such that $\epsilon_n=b_k-a_k$ and \[F(a_k+\epsilon_n)-F(a_k)= F(b_k)-F(a_k)\geq  2^{-(k-1)}m\epsilon_1=m\epsilon_n=p_n.\]
 The result follows from choosing $x_n$ equals to the $a_k$ such that $\epsilon_n=b_k-a_k$.
\end{proof}

\section{Proof of Proposition~\ref{prop:hypercube:nicefdp}}\label{sec:proofs:hypercube:nicefdp}

In this section we prove Proposition~\ref{prop:hypercube:nicefdp} by a coupling with another kind Bernoulli process.
Given a rooted graph $G$, the $\ell$-th level of $G$ is the set of vertices whose distance to the root is $\ell$.
Throughout this section the Bernoulli processes may have different probabilities at different levels:
we say that $\nu=\{\nu(v)\}_{v\in V(G)}$ is a Bernoulli process, if $\nu$ is a collection of i.i.d random variables such that vertices at the same level have the same probability to be open. Like in the other sections, the path $P=v_0\rightarrow v_1\rightarrow \cdots \rightarrow v_k$ is called open, if  $\nu=1$ for $i=0,1,\ldots,k$.

\begin{lem}\label{lem:2BernoulliHypercube}
 Let $n_0$ be a constant and take $n\geq n_0$. Let $\nu_{p_n,\widetilde{p},n_0}$ be a Bernoulli process on $\hycn{n}$ such that: $\rt$ and ${\bf 1}$ are open, there are $n_0$ levels whose vertices have $\widetilde{p}$ probability to be open, and vertices in other levels have probability $p_n$ to be open.
 Let $\ppberd{G}{p_n,\widetilde{p},n_0}$ be the probability of having a long open path in $\nu_{p,\widetilde{p},n_0}$, and
 let $Y=Y_{n,p_n}$ be the number of open paths from $\rt$ to ${\bf 1}$ in  $\nu_{p_n\widetilde{p},n_0}$. Then
 \begin{enumerate}
  \item $Y\sim\EE{Y}=n!\widetilde{p}^{n_0}p_n^{n-n_0-1}
  =\Omega(n!p_n^n).$
  \item If $np_n \rightarrow \infty$ as $n\rightarrow \infty$ then $Var(Y)=o(\EE{Y}^2)$.
  \item \label{lem:2BernoulliHypercube:3} If $np_n \rightarrow \infty$ as $n\rightarrow \infty$ then $\ppberd{\hycn{n}}{p_n,\widetilde{p},n_0}$ tends to one as $n\rightarrow \infty$.
 \end{enumerate}
\end{lem}

\begin{proof}
Let $\Aa{}$ denotes the set of paths in $\hycn{n}$ from $\rt$ to ${\bf 1}$.
As $|\Aa{}|=n!$ and each path in $\Aa{}$ is open with probability $\widetilde{p}^{n_0}p_n^{n-n_0-1}$ then
\[\mu=\EE{Y}=n!\widetilde{p}^{n_0}p_n^{n-n_0-1}.\]

Before we prove (1), we prove (2): $Var(Y)=o(\mu^2)$. Let define an indicator variables $Y_i$ for each path $i\in \Aa{}$ by
\[Y_i=
\begin{cases}
 1, & \text{ if the $i$-th path is open.}\\
 0, & \text{ other case.}
\end{cases}
\]

Fix any path $i_0$ in $\Aa{}$. For $0\leq k\leq n-1$, let $\Aa{k}$ denotes the set of paths in $\Aa{}$ that intersect $i_0$ in $k$ interior vertices. Thus

\begin{align*}
 Var(Y)&=\EE{Y^2}-\mu^2=-\mu^2+\sum_{i,j\in\Aa{}}\EE{Y_iY_j}
        =-\mu^2 + n!\sum_{j\in\Aa{}}\EE{Y_{i_0}Y_j} \\
    &=-\mu^2+ n!\sum_{j\in\Aa{0}}\EE{Y_{i_0}Y_j}
        +n!\sum_{k=1}^{n-1}\sum_{j\in\Aa{k}}\EE{Y_{i_0}Y_j}.
\end{align*}

Note that: if $j\in \Aa{0}$ then
\[\EE{Y_{i_0}Y_j}=\widetilde{p}^{2n_0}p_n^{2n-2n_0-2};\]
and, for $1\leq k\leq n-1$, if $j\in \Aa{k}$ then
\[\EE{Y_{i_0}Y_j}= p_n^{2n-2-k}  \tupla{\frac{\widetilde{p}}{p_n}}^{2n_0-s} \leq \tupla{\frac{\widetilde{p}}{p_n}}^{n_0} p_n^{2n-2-k},\]
where $0\leq s\leq \min\set{k,n_0}$ is the number of vertices in $i_{0}\cap j\cap L$.
Let $T(n,k)$ denotes the number of paths from $\rt$ to ${\bf 1}$ that intersect $i_0$ in exactly $k-1$ interior nodes. Thus

\begin{align*}
 Var(Y)&\leq-\mu^2+ n!T(n,1)\widetilde{p}^{2n_0}p_n^{2n-2n_0-2}
        +n!\tupla{\frac{\widetilde{p}}{p_n}}^{n_0}
        \sum_{k=1}^{n-1}T(n,k+1)p_n^{2n-2-k}\\
        &=-\mu^2+ \mu^2\frac{T(n,1)}{n!}
        +\mu^2\tupla{\frac{p_n}{\widetilde{p}}}^{n_0}
        \sum_{k=2}^{n}\frac{T(n,k)}{n!p_n^{k-1}}.
\end{align*}
In order to prove (2) we use the follows bounds for $T(n,k)$ obtained in \cite{Martinsson2014}. In Proposition~2.5 in \cite{Martinsson2014} it was proved that
\[
 n!(1-O(1/n))\leq T(n,1)\leq n!.
\]
 In Equation~3.4 of \cite{Martinsson2014} it was proved that, if $np_n\rightarrow \infty$ as $n\rightarrow \infty$, then
\begin{equation}
 \sum_{k=2}^n\frac{T(n,k)}{n!p_n^{k-1}}=o(1)
\end{equation}

from which
\[
Var(Y) \leq -\mu^2+ \mu^2 %(1-O(1/n))
        +\mu^2\tupla{\frac{p_n}{\widetilde{p}}}^{n_0} o(1)=o(\mu^2).
\]

For proving $(3)$ (as it was proved in Lemma~2.2 of \cite{Martinsson2014}), recall that if $X$ is a random variable with finite expected value and finite nonzero second moment, then
\[\pr{X\neq 0}\geq \frac{\EE{X}^2}{\EE{X^2}}. \]
Thus, as
\[\pr{Y\neq 0}\geq \frac{\EE{Y}^2}{\EE{Y^2}}
= \frac{\EE{Y}^2}{\EE{Y}^2 +o(\EE{Y}^2)},\]
then $\lim_{n \rightarrow \infty} \pr{Y\geq 1}=1$.

Proof of (1) follows from
\[
 \pr{\left|\frac{Y}{\mu}-1\right|<\epsilon } \geq 1-\frac{Var(Y)}{\epsilon^2\mu^2}=1-o(1).
\]

\end{proof}

\begin{proof}[\bf Proof of Proposition~\ref{prop:hypercube:nicefdp}]
 Let $f$, $X$ and $\omega_{\eta,\theta}$ be as in Proposition~\ref{prop:hypercube:nicefdp}.
 Let $I_{\max}$ be an interval with length $\theta$ such that \[C=\int_{I_{\max}}f(x)d(x).\]
 Let $n_0$ be such that $\theta\cdot n_0 /2$ is greater than the length of $Supp(f)$; without loss of generality we assume that $n>n_0+3$.
 Let \[C_{\min}=\min\set{\int_{I}f(x)d(x): \text{ length}(I)=\theta/2},\]  where $I$ denotes a closed interval.
 As $f$ is continuous on $Supp(f)$ and $Supp(f)$ is connected and bounded, then $C_{\min}>0$.

 In the following we construct Bernoulli processes, $\zeta$ and $\nu_{C,C_{\min},n_0+3}$ on $\hycn{n}$, where $\nu_{C,C_{\min},n_0+3}$ is as in Lemma~\ref{lem:2BernoulliHypercube} and $\zeta$ is in the same probability space than $\omega_{\eta,\theta}$. Then we prove that $\stdom{\zeta}{\nu_{C,C_{\min},n_0+3}}$ and  open paths in $\zeta$ are accessible paths in $\omega_{\eta,\theta}$. Thus, by Lemma~\ref{lem:2BernoulliHypercube}, the number of open paths in $\zeta$ is $\sim \Omega(n! C^{n-1})$ and \[X \sim \Omega(n!C^{n-1}).\]

 See Figure~\ref{fig:Intervals} and note that there exist a sequence of intervals $I_1,I_2,\ldots,I_{n_0}$, each one with length $\theta/2$, such that:
 \begin{itemize}
  \item The support of $f$ is equals to $\bigcup_{i=1}^{n_0}{I_i}$.
  \item For $1\leq i\leq n_0-1$, if $x_i\in I_i$ and $x_{i+1}\in I_{i+1}$ then $|x_{i+1}-x_i|< \theta$.
  \item There is no an $x$ in the $supp(f)\setminus I_{n_0}$ to the left of $I_{n_0}$, and there is no an $x$ in the $Supp(f)\setminus I_{1}$ to the right of $I_{1}$.
  \item There is a $k$ such that $I_k\subseteq I_{\max}\subseteq I_{k+1}\cup I_k \cup I_{k-1}$ and $I_{k+1}, I_k, I_{k-1}$ are in this order from left to right.
 \end{itemize}

 \begin{figure}[htb]
	\centering
	\includegraphics[width=.7\linewidth]{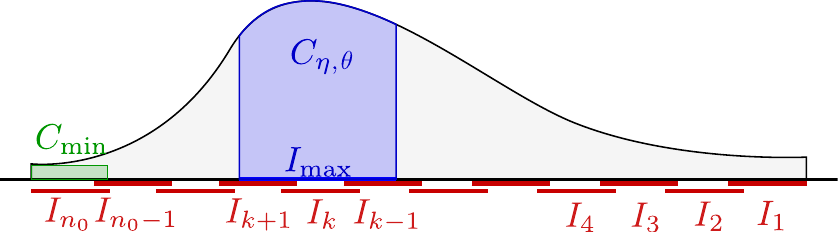}
	\caption{An illustration of $C_{\eta,\theta}$, $C_{\min}$ and intervals in proof of Proposition~\ref{prop:hypercube:nicefdp}.}
	\label{fig:Intervals}
 \end{figure}%

  Let $\zeta$ be a Bernoulli process on $\hycn{n}$, where, for each vertex $v_d$ at distance $d$ to $\rt$,
\[\zeta(v_d)=\begin{cases}
          1 & \text{ if } d=0 \text{ or } d=n,\\
          1 & \text{ if } \eta(v_d)\in I_{d} \text{ and } 1\leq d\leq k+1,\\
          1 & \text{ if } \eta(v_d)\in I_{\max} \text{ and } k+2\leq d\leq n-(n_0-k+2),\\
          1 & \text{ if } \eta(v_d)\in I_{d-(n-n_0)} \text{ and } n-(n_0-k+1) \leq d<n,\\
          0 & \text{ otherwise.}  \\
         \end{cases}
\]

 Let $P=v_0\rightarrow v_1\rightarrow \cdots \rightarrow v_n $ be an open path in $\zeta$. We claim that $P$ is accessible in $\omega_{\eta,\theta}$.
 \begin{itemize}
  \item As there is no an $x$ in the $Supp(f)\setminus I_{1}$ to the right of $I_{1}$, $\eta(v_0)$ is in $I_1$ or is to the left of $I_{1}$. Thus, whatever the value of $\eta(v_0)$,  $\eta(v_{1})-\eta(v_0)>-\theta/2$ and
  \begin{align*}
        \omega_{\eta,\theta}(v_{1})-\omega_{\eta,\theta}(v_{0})=&\eta(v_{1})-\eta(0)+\theta\cdot d(v_{1})- \theta\cdot d(v_0)\\
        =& \eta(v_{1})-\eta(v_0)+\theta\\
        \geq& -\theta/2+\theta>0.
  \end{align*}

   \item Let $1\leq i \leq n-2$. As edges in $\hycn{n}$ always point to vertices further away from a vertex $v_0$, $d(v_{i+1})=d(v_{i})+1$. As any two vertices in $I_i$ and $I_{i+1}$ are at distance smaller than  $\theta$, then  $\eta(v_{i+1})-\eta(v_i)>-\theta$. Thus
   \begin{align*}
        \omega_{\eta,\theta}(v_{i+1})-\omega_{\eta,\theta}(v_{i})=&\eta(v_{i+1})-\eta(v_i)+\theta\cdot d(v_{i+1})- \theta\cdot d(v_i)\\
        =& \eta(v_{i+1})-\eta(v_i)+\theta\\
        >& -\theta+\theta=0.
   \end{align*}

  \item As there is no an $x$ in the $supp(f)\setminus I_{n_0}$ to the left of $I_{n_0}$, $\eta(v_n)$ is in $I_{n_0}$ or is to the right of $I_{n_0}$. Thus, whatever the value of $\eta(v_n)$,  $\eta(v_{n})-\eta(v_{n-1})>-\theta/2$ and
  \begin{align*}
        \omega_{\eta,\theta}(v_{n})-\omega_{\eta,\theta}(v_{n-1})=&\eta(v_{n})-\eta(n-1)+\theta\cdot d(v_{n})- \theta\cdot d(v_{n-1})\\
        =& \eta(v_{n})-\eta(v_{n-1})+\theta\\
        \geq& -\theta/2+\theta>0.
  \end{align*}
\end{itemize}

Thus, open paths in $\zeta$ implies accessible paths in $\omega_{\eta,\theta}$.

See the definition of $\nu_{p_n\widetilde{p},n_0}$ in Lemma~\ref{lem:2BernoulliHypercube} and note that
\begin{align*}
 \pr{\zeta(v_d)=1} &=\begin{cases}
          1 & \text{ if } d=0 \text{ or } d=n,\\
          \pr{\eta(v_d)\in I_{d}} &  \text{ if }  1\leq d\leq k+1,\\
          \pr{\eta(v_d)\in I_{\max}} & \text{ if }   k+2\leq d\leq n-(n_0-k+2),\\
          \pr{\eta(v_d)\in I_{d-(n-n_0)}} &  \text{ if }  n-(n_0-k+1) \leq d<n.\\
         \end{cases}\\
         & \geq \begin{cases}
          1 & \text{ if } d=0 \text{ or } d=n,\\
          C_{\min} &  \text{ if }  1\leq d\leq k+1 \text{ or } n-(n_0-k+1) \leq d<n, \\
          C & \text{ if }   k+2\leq d\leq n-(n_0-k+2).
         \end{cases}\\
         & = \pr{\nu_{C,C_{\min},n_0+3}(v_d)=1}.
\end{align*}

Thus $\stdom{\zeta}{\nu_{C,C_{\min},n_0+3}}$ and the proof follows as above.

\end{proof}

\section{Results on infinite graphs}\label{sec:inf:graphs}

In this section we show our results on infinite graphs. First we prove theoretical results on  $\Zd$ and $\Zda$. Then we show experimental results on regular trees, $\Zd$ and $\Zda$.

\subsection{Theoretical results on infinite graphs}

 Let $G$ be an infinite graph and $\nu_p$ be a Bernoulli process on $G$.
 Recall that, in this case, long paths means paths that starts at the source and  are infinite.
 Like the finite case, we denote by
\[\ppberd{G}{p}=\text{Probability of having a long open path in } \nu_p.\]
 Unlike the finite case, we say that there is \emph{Bernoulli percolation} in $(G,p)$, if $\ppberd{G}{p}>0$.
 Note that Lemma~\ref{lem:basic_coupling} also holds on infinite graphs. In this section, we use Lemma~\ref{lem:basic_coupling} and some known bounds for site percolation on $\Zd$ and $\Zda$, in order to obtain bounds for $\theta$ for having RMF percolation on $\Zd$ and $\Zda$.

\begin{cor} \label{cor:L2}
 Let $F$ be any distribution and $\theta>0$. If  $\maxC{F}{\theta}>0.75$, then
 \[\pprmf{\Zd}{F}{\theta}>0 \hspace{1 em} \text{ and }  \hspace{1 em}  \pprmf{\Zda}{F}{\theta}>0.\]
\end{cor}

Results in Corollary~\ref{cor:L2} follow from Theorem~\ref{thm:coupling} and the bounds for site percolation on $\Zd$ and $\Zda$ obtained in \cite{BoundsZ2percliggett1995}: if $p>0.75$ then $\ppberd{\Zda}{p}>0$ and $\ppberd{\Zd}{p}>0$.

\subsection{Experimental results on infinite graphs}

We performed Monte Carlo simulations of the RMF accessibility percolation on $2$-regular trees, $3$-regular trees, $\Zd$ and $\Zda$. The labels of nodes were taken from a \emph{Uniform(0,1)} distribution, and the values of $\theta$ were varied from $0$ to $1$ in steps of $0.001$. Twenty thousand Monte Carlo runs were performed for each value of $\theta$.
Those graphs were tested until hight $125$, $250$, $500$, $1000$ and $2000$.

\subsubsection*{Simulations on regular Trees}
Let $\rtree{n}$ denotes the $n$-regular tree and $U(0,1)$ denotes the \emph{Uniform(0,1)} distribution. The result of these simulations are Illustrated in Figure~\ref{fig:Simul_2_3_ary_unif}.
 \begin{figure}[htb]
	\centering
	\includegraphics[width=.38\linewidth]{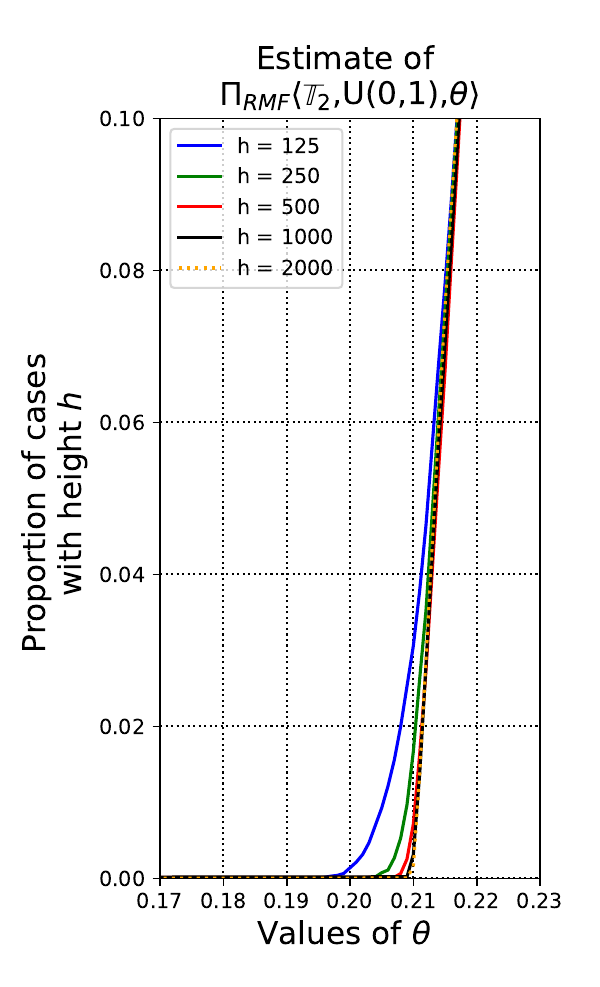}\vspace{1cm}
	\includegraphics[width=.38\linewidth]{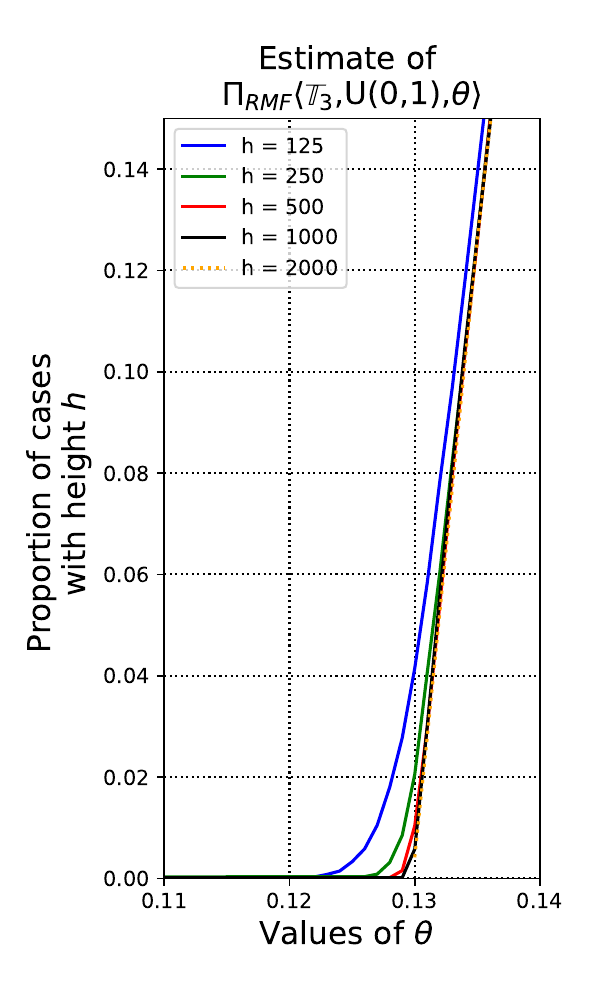}
	\caption{Simulations on $2$-regular trees and $3$-regular trees where nodes have distribution \emph{Uniform(0,1)}. In both cases, for different values of $\theta$, it is illustrated the proportion of cases for which there is an accessible path with determined height. Twenty thousand Monte Carlo runs were performed for each value of $\theta$.}
	\label{fig:Simul_2_3_ary_unif}
 \end{figure}%

In the case of $2$-regular trees: for $\theta<0.2$, there are no tracks that reach hight $250$; for $\theta>0.22$ there are tracks that reach hight $2.000$. It suggest that
\[\pprmfd{\rtree{2}}{U(0,1),0.2}=0 \text{ and }\pprmfd{\rtree{2}}{U(0,1),0.22}>0.\]
Similarly, for the case of $3$-regular trees we have that
\[\pprmfd{\rtree{3}}{U(0,1),0.12}=0 \text{ and }\pprmfd{\rtree{3}}{U(0,1),0.14}>0.\]

In comparison to the critical behavior of Bernoulli percolation: on $\rtree{2}$ it is known that $p_c=1/2$ \cite{ubralArboles}, but for RMF accessibility percolation it should be between $0.2$ and $0.22$; on $\rtree{3}$ it is known that  $p_c=1/3$ \cite{ubralArboles}, but for RMF accessibility percolation it should  be between $0.12$ and $0.14$.

\subsubsection*{Simulations on $\Zd$ and $\Zda$}

The result simulations on $\Zd$ and $\Zda$ are Illustrated in Figure~\ref{fig:Simul_L2_L2alt_unif}.
 \begin{figure}[htb]
	\centering
	\includegraphics[width=.38\linewidth]{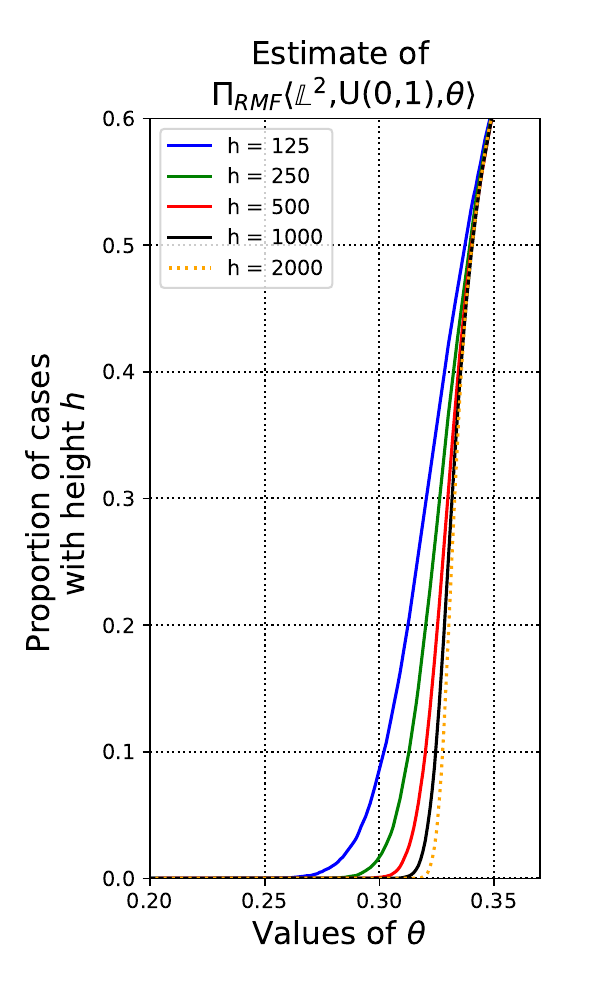}\vspace{1cm}
	\includegraphics[width=.38\linewidth]{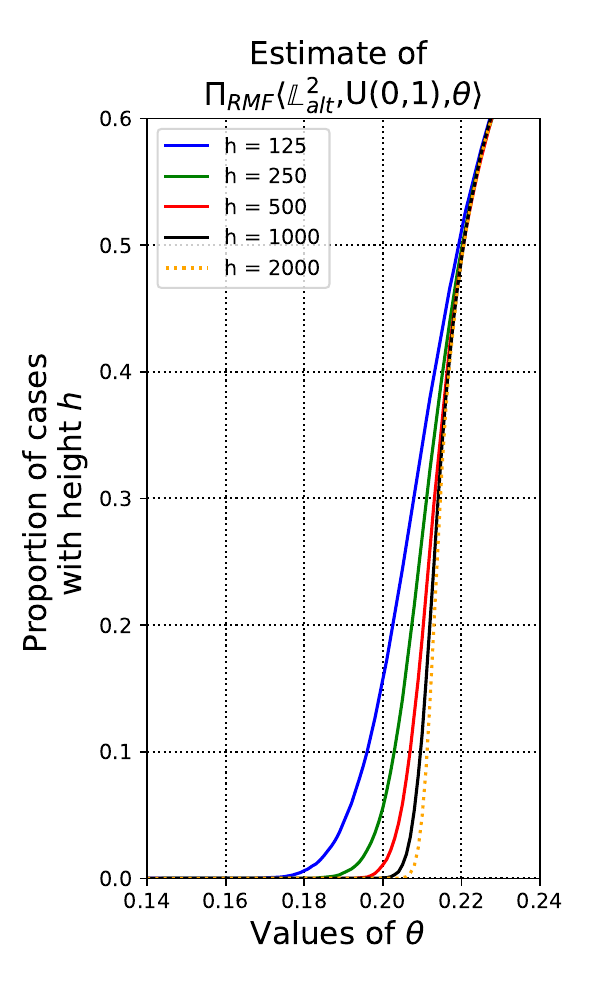}
	\caption{Simulations on $\Zd$ and $\Zda$ where nodes have distribution \emph{Uniform(0,1)}. In both cases, for different values of $\theta$, it is illustrated the proportion of cases for which there is an accessible path with determined height. Twenty thousand Monte Carlo runs were performed for each value of $\theta$.}
	\label{fig:Simul_L2_L2alt_unif}
 \end{figure}%

In the case of $\Zd$: for $\theta<0.3$, there are no tracks that reach hight $1000$; for $\theta>0.33$ there are tracks that reach hight $2.000$. It suggest that
\[\pprmfd{\Zd}{U(0,1),0.3}=0 \text{ and }\pprmfd{\Zd}{U(0,1),0.33}>0.\]
Similarly, for the case of $\Zda$ we have that
\[\pprmfd{\Zda}{U(0,1),0.19}=0 \text{ and }\pprmfd{\Zda}{U(0,1),0.22}>0.\]

In comparison to the critical behavior of Bernoulli percolation: on $\Zd$ it is obtained in $p_c=0.7054$  \cite{Jensen}, but for RMF accessibility percolation it should be between $0.3$ and $0.33$; on $\Zda$ it is obtained in $p_c=0.5956$  \cite{Jensen}, but for RMF accessibility percolation it should be  between $0.19$ and $0.22$.

\noindent\textbf{Acknowledgments:}  The authors would like to thank  the anonymous referees for their careful reading, criticism and suggestions which helped us to considerably improve the paper.

%%%%%%%%%%%%%%%%%%%%%%%%%%%%%%%%%%%%%%%%%%%%%%%%%%%%%%%%%%%%%%%%%%%%%%%%%%%
\bibliographystyle{plain} \bibliography{CouplingRMF.bib}

\begin{thebibliography}{10}

\bibitem{RMF-DEF}
Takuyo Aita, Hidefumi Uchiyama, Tetsuya Inaoka, Motowo Nakajima, Toshio Kokubo,
  and Yuzuru Husimi.
\newblock Analysis of a local fitness landscape with a model of the rough mt.
  fuji-type landscape: application to prolyl endopeptidase and thermolysin.
\newblock {\em Biopolymers.}, 54(1):64--79, 2000.

\bibitem{10.2307/43864013}
Julien Berestycki, \'{E}ric Brunet, and Zhan Shi.
\newblock The number of accessible paths in the hypercube.
\newblock {\em Bernoulli}, 22(2):653--680, 2016.

\bibitem{JKAJ2011}
Jasper Franke, Alexander Klözer, J.~Arjan G.~M. de~Visser, and Joachim Krug.
\newblock Evolutionary accessibility of mutational pathways.
\newblock {\em PLOS Computational Biology}, 7(8):1--9, 2011.

\bibitem{FK2012}
Jasper Franke and Joachim Krug.
\newblock Evolutionary accessibility in tunably rugged fitness landscapes.
\newblock {\em Journal of Statistical Physics}, 148(4):706--723, 2012.

\bibitem{JGJ2010}
Jasper Franke, Gregor Wergen, and Joachim Krug.
\newblock Records and sequences of records from random variables with a linear
  trend.
\newblock {\em Journal of Statistical Mechanics: Theory and Experiment},
  2010(10):P10013, 2010.

\bibitem{DirPercZd2002}
Geoffrey Grimmett and Philipp Hiemer.
\newblock {\em Directed Percolation and Random Walk}, pages 273--297.
\newblock Birkh{\"a}user Boston, Boston, MA, 2002.

\bibitem{Martinsson2014}
Peter Hegarty and Anders Martinsson.
\newblock On the existence of accessible paths in various models of fitness
  landscapes.
\newblock {\em Ann. Appl. Probab.}, 24(4):1375--1395, 2014.

\bibitem{Jensen}
Iwan Jensen.
\newblock Low-density series expansions for directed percolation on square and
  triangular lattices.
\newblock {\em Journal of Physics A: Mathematical and General}, 29(22):7013,
  1996.

\bibitem{kingman1978}
John~FC Kingman.
\newblock A simple model for the balance between selection and mutation.
\newblock {\em Journal of Applied Probability}, 15(1):1--12, 1978.

\bibitem{krug2021accessibility}
Joachim Krug.
\newblock Accessibility percolation in random fitness landscapes.
\newblock In {\em Probabilistic Structures in Evolution}, pages 1--22. E. Baake
  and A. Wakolbinger {EMS} Press, 2021.

\bibitem{Li2017}
Li~Li.
\newblock Phase transition for accessibility percolation on hypercubes.
\newblock {\em Journal of Theoretical Probability}, 31(4):2072--2111, 2018.

\bibitem{BoundsZ2percliggett1995}
Thomas~M Liggett.
\newblock Survival of discrete time growth models, with applications to
  oriented percolation.
\newblock {\em The Annals of Applied Probability}, pages 613--636, 1995.

\bibitem{ubralArboles}
Russell Lyons.
\newblock {Random Walks and Percolation on Trees}.
\newblock {\em The Annals of Probability}, 18(3):931 -- 958, 1990.

\bibitem{Neidhart2014}
Johannes Neidhart.
\newblock {\em Fitness Landscapes, Adaptation and Sex on the Hypercube}.
\newblock PhD thesis, Universit{\"a}t zu K{\"o}ln, 2014.

\bibitem{NK2013}
S.~Nowak and J.~Krug.
\newblock Accessibility percolation on n-trees.
\newblock {\em EPL (Europhysics Letters)}, 101(6):66004, 2013.

\bibitem{ParkKrug2008}
Su-Chan Park and Joachim Krug.
\newblock Evolution in random fitness landscapes: the infinite sites model.
\newblock {\em Journal of Statistical Mechanics: Theory and Experiment},
  2008(04):P04014, 2008.

\bibitem{RZ2013}
Matthew Roberts and Lee Zhao.
\newblock Increasing paths in regular trees.
\newblock {\em Electronic Communications in Probability}, 18:1--10, 2013.

\bibitem{SMK2013}
Ivan Szendro, Martijn Schenk, Jasper Franke, Joachim Krug, and J~Arjan
  de~Visser.
\newblock Quantitative analyses of empirical fitness landscapes.
\newblock {\em Journal of Statistical Mechanics Theory and Experiment},
  2013(01):P01005, 2013.

\bibitem{thorisson2000coupling}
H.~Thorisson.
\newblock {\em Coupling, Stationarity, and Regeneration}.
\newblock Probability and Its Applications. Springer New York, 2000.

\bibitem{WDDH2006}
Daniel~M. Weinreich, Nigel~F. Delaney, Mark~A. DePristo, and Daniel~L. Hartl.
\newblock Darwinian evolution can follow only very few mutational paths to
  fitter proteins.
\newblock {\em Science}, 312(5770):111--114, 2006.

\bibitem{WWC2005}
Daniel~M. Weinreich, Richard~A. Watson, and Lin Chao.
\newblock Perspective: Sign epistasis and genetic constraint on evolutionary
  trajectories.
\newblock {\em Evolution; international journal of organic evolution},
  59(6):1165--1174, 2005.

\end{thebibliography}
%%%%%%%%%%%%%%%%%%%%%%%%%%%%%%%%%%%%%%%%%%%%%%%%%%%%%%%%%%%%%%%%%%%%%%%%%%%

\begin{appendices}
\section{Notation}\label{sec:notation}

Throughout this paper, we will employ the following notation. Let $g,h: \NN \rightarrow \RR^+$ be any two functions.
\begin{enumerate}[label=(\roman*)]
 \item $g(n) \sim h(n)$ means that $\lim_{n\rightarrow \infty} \frac{g(n)}{h(n)}=1$.
 \item $g(n)\gtrsim h(n)$ means that $\limsup_{n \rightarrow \infty} \frac{g(n)}{h(n)} \geq 1$.
 \item $g(n)\lesssim h(n)$ means that $h(n) \gtrsim g(n)$.
 \item $g(n)=O(h(n))$ means that $\limsup_{n \rightarrow \infty } \frac{g(n)}{h(n)} < \infty$.
 \item $g(n)=\Omega (h(n))$ means that $h(n)=O(g(n))$.
 \item $g(n)=o(h(n))$ means that $\lim_{n \rightarrow \infty} \frac{g(n)}{h(n)}=0$.
\end{enumerate}

Let $\set{g(n)}_{n=1}^\infty$, $\set{h(n)}_{n=1}^\infty$ be two sequences of random variables.
\begin{enumerate} [label=(\roman*)]
\setcounter{enumi}{6}
 \item We write $g(n) \sim h(n)$ if, for all $\epsilon_1, \epsilon_2 >0$ and $n$ sufficiently long,
 \[\pr{\left|\frac{g(n)}{h(n)} -1 \right| < \epsilon_1} > 1-\epsilon_2. \]
 \item We write $g(n) \gtrsim h(n)$ if, for all $\epsilon_1, \epsilon_2 >0$ and $n$ sufficiently long,
 \[\pr{\frac{g(n)}{h(n)} > 1-\epsilon_1} > 1-\epsilon_2. \]
\end{enumerate}
\end{appendices}

\end{document}